\documentclass{tran-l}

\usepackage{tikz}
\usepackage{enumerate}
\usepackage{amssymb}
\usepackage{amsmath}
\usepackage{mathrsfs}
\usepackage{amsthm}
\usepackage{mathtools}
\usepackage[colorlinks,linkcolor=blue]{hyperref}

\newtheorem{theorem}{Theorem}[section]
\newtheorem{lemma}[theorem]{Lemma}

\theoremstyle{definition}

\newtheorem{exm}[theorem]{Example}
\newtheorem{problem}[theorem]{Problem}
\newtheorem{proposition}[theorem]{Proposition}
\newtheorem{corollary}[theorem]{Corollary}

\newcommand{\Ric}{\operatorname{Ric}}
\newcommand{\R}{\mathbb{R}}

\theoremstyle{remark}
\newtheorem{remark}[theorem]{Remark}

\numberwithin{equation}{section}

\begin{document}

\title[Llarull type theorems on complete manifolds with PSC]{Llarull type theorems on complete manifolds with positive scalar curvature}


\author{Tianze Hao}
\address{Key Laboratory of Pure and Applied Mathematics,
School of Mathematical Sciences, Peking University, Beijing, 100871, P. R. China
}
\email{haotz@pku.edu.cn}

\author{Yuguang Shi}
\address{Key Laboratory of Pure and Applied Mathematics,
School of Mathematical Sciences, Peking University, Beijing, 100871, P. R. China
}
\email{ygshi@math.pku.edu.cn}

\author{Yukai Sun}
\address{Key Laboratory of Pure and Applied Mathematics,
School of Mathematical Sciences, Peking University, Beijing, 100871, P. R. China
}
\email{sunyukai@math.pku.edu.cn}
\thanks{T. Hao, Y. Shi and Y. Sun are funded by the National Key R\&D Program of China Grant 2020YFA0712800.}



\date{}

\dedicatory{Scalar curvature, $\mu$-bubble, Llarull's Theorem}

\begin{abstract}
In this paper, without   assuming that manifolds  are spin,  we prove that if a compact  orientable, and connected Riemannian  manifold $(M^{n},g)$ with scalar curvature $R_{g}\geq 6$ admits a non-zero degree and $1$-Lipschitz map to $(\mathbb{S}^{3}\times \mathbb{T}^{n-3},g_{\mathbb{S}^{3}}+g_{\mathbb{T}^{n-3}})$, for $4\leq n\leq 7$, then $(M^{n},g)$ is locally isometric to $\mathbb{S}^{3}\times\mathbb{T}^{n-3}$. Similar results are established for noncompact cases as $(\mathbb{S}^{3}\times \mathbb{R}^{n-3},g_{\mathbb{S}^{3}}+g_{\mathbb{R}^{n-3}})$ being model spaces (see Theorem \ref{noncompactrigidity1}, Theorem \ref{noncompactrigidity2}, Theorem \ref{noncompactrigidity3}, Theorem \ref{noncompactrigidity4}).
We observe that the results differ significantly when $n=4$ compared to $n\geq 5$.  Our results imply that the $\epsilon$-gap length extremality of the standard $\mathbb{S}^3$ is stable under the Riemannian product with $\mathbb{R}^m$, $1\leq m\leq 4$ (see $D_{3}$. Question in Gromov's paper \cite{Gromov2017}, p.153). \end{abstract}

\maketitle
\tableofcontents
\section{Introduction}

 It is commonly believed that for some Riemannian manifolds $(M^n, g)$ with positive scalar curvature (PSC) metrics, it is not possible to increase the scalar curvature and while simultaneously enlarging the manifolds in all directions. The main results in work of Llarull \cite{LM1998} and Gromov and Zhu \cite{gromov2021area} align with this notion, more detailed discussions can be found in a survey paper of Gromov \cite{Gr2023} and references therein.  In this paper, we aim to investigate similar phenomena observed in some manifolds endowed with PSC metrics. For instance, if we consider the  manifold $(\mathbb{S}^3\times N^{n-3}, g)$ with a metric $g$ with its scalar curvature $R_g\geq 6$,  here $N^{n-3}$ is either $\mathbb{T}^{n-3}$ or $\mathbb{R}^{n-3}$ which is enlargeable in the sense of Gromov and Lawson \cite{GL1980}. Then according to the reasons described above, we know the factor $\mathbb{S}^3$ with the induced metric should not spread in all direction. Indeed, the following theorem can be regarded as an evidence to support this assertion,  namely, we want to show the following:

\begin{theorem}\label{comprigidity1}	
Let $(M^n, g)$ be a $n$-dimensional compact(without boundary) complete orientable connected Riemannian manifold with $R_g\geq 6$, $4\leq n\leq 7$,  we assume that  there is a non-zero degree and $1$-Lipschitz  map $f:M^n\to \mathbb{S}^3\times \mathbb{T}^{n-3}$,	then  $(M^n, g)$ is locally isometric to $\mathbb{S}^3\times \mathbb{T}^{n-3}$. Moreover,  for any $x\in\mathbb{T}^{n-3} $, $P \circ f(\cdot,x): \mathbb{S}^3 \to \mathbb{S}^3 $ is isometric, here $P$ denotes the standard projection of $P:\mathbb{S}^3\times \mathbb{T}^{n-3} \to \mathbb{S}^3 $.
\end{theorem}
\begin{remark}
    In fact, in Theorem \ref{comprigidity1}, we  are able to demonstrate  the standard unit sphere $\mathbb{S}^{3}$ can be isometrically embedding into $(M^{n},g)$, and there is a tubular neighborhood of this $\mathbb{S}^{3}$ in $M^{n}$ which is isometric to that of $\mathbb{S}^{3}$ in $\mathbb{S}^3\times \mathbb{T}^{n-3}$.
\end{remark}

 An affirmative answer to the following problem would be a direct generalization of Theorem \ref{comprigidity1} to  noncompact cases.

 \begin{problem}\label{noncompactprob1}
Let $(M^n, g)$ be a  noncompact orientable connected and complete Riemannian manifold with scalar curvature $R_g\geq 6 $, $f:M^n\to \mathbb{S}^3\times \mathbb{R}^{n-3}$ be a proper  and $1$-Lipschitz  map with non-zero degree.	 Is $(M^n, g)$ (locally) isometric to $\mathbb{S}^3\times \mathbb{R}^{n-3}$?
\end{problem}

However, the noncompact situations are much more complicated than those of compact cases. Indeed, the answer to Problem \ref{noncompactprob1} may be negative if $n\geq 5$. More specifically, we have the following example:

\begin{exm}\label{example1}
Let $\Sigma^{n}$, $n\geq 2$,  be a paraboloid of revolution in $\mathbb{R}^{n+1}$ which is also a graph of $\mathbb{R}^{n}\subset\mathbb{R}^{n+1}$, $P:\Sigma^{n}\to \mathbb{R}^{n}$ denotes the restriction of the standard projection in $\mathbb{R}^{n+1}$ to $\mathbb{R}^{n}$. Let $(M^{n+m},g)$ be the Riemannian product of $\Sigma^{n} $ and $\mathbb{S}^m$, $m\geq 2$, note that its scalar curvature  $R_g>m(m+1)$. Then $f:=(id, P):M^{n+m}\to \mathbb{S}^m\times \mathbb{R}^{n} $ is a 	proper  and $1$-Lipschitz  map with non-zero degree.\end{exm}

Example \ref{example1} illustrates  that, in $\times \mathbb{R}^m$-Stabilized Mapping Theorem in p.139 in Gromov's survey paper \cite{Gr2023},  there could be no $x$ in $M$ with
$$R_g(x)\leq \|\Lambda^2df(x)\|\cdot R_{\bar g}(f(x)),$$
while there holds (c.f.Non-compact Extremality Theorem in p.317 of \cite{Gr2023} )
$$\inf_{x\in M}(R_g(x)-\|\Lambda^2df(x)\|\cdot R_{\bar g}(f(x)))=0, $$
for definition of $\Lambda^2df(x)$, see p.126 in \cite{Gr2023}.

 While the answer to the Problem \ref{noncompactprob1} is affirmative under an extra condition when $n=4$. i.e., we are able to show that

\begin{theorem}\label{noncompactrigidity1}
Let $(M^4, g)$ be a  noncompact orientable connected and complete Riemannian manifold with scalar curvature $R_g\geq 6 $, $f:M^4\to \mathbb{S}^3\times\mathbb{R}$ be a proper  and $1$-Lipschitz  map with non-zero degree.	 Then $(M^4, g)$ is isometric to $\mathbb{S}^3\times\mathbb{R}$ provided $(M^4, g)$ is of bounded geometry, i.e. $\sup_{M}\|Rm\|<\infty$ and its injective radius $inj(M^n, g)>0$.
\end{theorem}

 In contrast to Example \ref{example1}, the  theorem below gives a completely different pictures for a $4$-dimensional noncompact manifold with a PSC metric. To be precise:
\begin{theorem}\label{noncompactrigidity2}
Let $(M^4, g)$ be a  noncompact orientable connected and complete Riemannian manifold with scalar curvature $R_g> 6 $, then there is no proper  and $1$-Lipschitz  map  $f:M^4\to \mathbb{S}^3\times\mathbb{R}$ with non-zero degree.
\end{theorem}

  It turns out a similar result still holds when the scalar curvature $R_g\geq 6+\delta$ for $5\leq n\leq 7$, i.e.

\begin{theorem}\label{noncompactrigidity3}
For any $\delta>0$, let $(M^n, g)$, $5\leq n\leq 7$, be a  noncompact orientable connected and complete Riemannian manifold with scalar curvature $R_g\geq 6+\delta $.  Then there is no proper  and $1$-Lipschitz  map  $f:M^n\to \mathbb{S}^3\times\mathbb{R}^{n-3}$ with non-zero degree.
\end{theorem}

Theorem \ref{noncompactrigidity2} and Theorem \ref{noncompactrigidity3},    in conjunction with Llarull's result demonstrate that for any $\epsilon>0$,  $\epsilon$-gap length extremality of the standard $\mathbb{S}^3$ is stable under the Riemannian product with $\mathbb{R}^m$, $1\leq m\leq 4$ (see $D_{3}$. Question in a paper of Gromov \cite{Gromov2017}, p.153).  Owing to Theorem \ref{noncompactrigidity2}, it is our contention that the hypothesis of $ (M^4,g) $ exhibiting bounded geometry as stated in Theorem  \ref{noncompactrigidity1}  is unnecessary.  According to  Example \ref{example1}, we see that the condition $R_g\geq 6+\delta $ can not be removed. On the other hand, it is natural to study the rigidity for a noncompact manifold in Problem \ref{noncompactprob1}. Again, by Example \ref{example1}, it becomes clear that additional conditions are imperative. To elucidate further:
\begin{theorem}\label{noncompactrigidity4}
Let $(M^n, g)$,  $4\leq n\leq 7$,  be a  noncompact orientable connected and complete Riemannian manifold with scalar curvature $R_g\geq 6 $, $f:M^n\to \mathbb{S}^3\times\mathbb{R}^{n-3}$ be a proper  and $1$-Lipschitz  map with non-zero degree.	 Then $(M^n, g)$ is  isometric to $\mathbb{S}^3\times\mathbb{R}^{n-3}$ provided $f$ is isometric outside a compact domain of $M^n$.
\end{theorem}

There is a relevant problem for $f:(M^n, g)\to \mathbb{S}^{n}$, where $(M^n, g)$ is a complete manifold with $R_g\geq n(n-1)$, and $f$ is a proper  and area decreasing map with non-zero degree      was raised  by M. Gromov (p.254 in \cite{Gr2023} ). When $M^n$ is noncompact and spin, an interesting result was obtained by Zhang \cite{Zhangw2020}.

In order to substantiate our theorems, it is noted that, utilizing  dimension reduction arguments initiated  by Fischer-Colbrie and Schoen \cite{FS1980}, Schoen and Yau\cite{ScY1979}, for any Riemannian manifold  $(M^n, g)$  referenced in the aforementioned theorems, one can derive a $3-$dimensional compact and orientable  manifold $\Sigma^3$, a crucial step entails verifying that such a manifold $\Sigma^3$  equip with the induced metric is isometric to the standard round $\mathbb{S}^3$. In order to achieve this, we employ  torical symmetrization techniques pioneered by Schoen and  Yau \cite{SY1}, Gromov and Lawson \cite{GL1983}, among others, in conjunction with methods involving Dirac operator. It is important to observe that  $\Sigma^3$ is automatically spin and so is the resulting manifold from $\Sigma^3$  warped product with several circles $\mathbb{S}^1$.  Consequently, we can apply the methodologies detailed in \cite{LM1998} to get what we want,  for further information, refer to Proposition \ref{Prop-rigid} below. Once it is established that $\Sigma^3$, with the induced metric, is isometric to the standard round $\mathbb{S}^3$, foliation techniques can be utilized to deduce the definitive rigidity results. We emphasize that, contrary to the typical scenarios where the Dirac operator is utilized, there is no requirement in the aforementioned theorems for the manifolds to be spin, nor is there a necessity for any specialized structures at the infinity of manifolds in Theorem \ref{noncompactrigidity1}, Theorem \ref{noncompactrigidity2}, Theorem \ref{noncompactrigidity3}.

Following the idea in a work of Zhu \cite{ZJT2020}, we construct a  $3$-dimensional compact orientable hypersurface $\Sigma^3 $ referred to as  $\mu$-bubble in  $(M^4,g)$ in Theorem \ref{noncompactrigidity1}.  Specifically, we utilize function $h_\epsilon$  introduced in Lemma 2.3 in  \cite{ZJT2020}  to establish a $h_\epsilon$-bubble  in  $(M^4,g)$. By leveraging the bounded geometry assumption, we  establish the convergence of   those  $h_\epsilon$-bubbles to the aforementioned $\Sigma^3 $,  which is identified as  the standard sphere $\mathbb{S}^3$.

Suppose there exists a  proper  and $1$-Lipschitz  map  $f:M^4\to \mathbb{S}^3\times\mathbb{R}$ with non-zero degree in Theorem \ref{noncompactrigidity2}.  Drawing on the proof strategy of  Theorem 3.5 in work of Cecchini, R\"{a}de and Zeidler \cite{Ce_2023}, we  derive a $\mu-$bubble $\Sigma^3 $ in $(M^4, g)$ in Theorem \ref{noncompactrigidity2} for  a carefully chosen function $\mu$, and we  demonstrate that  the scalar curvature of the manifold $(\bar M^4, \bar g)$ defined by warped product of $\Sigma^3 $ with a circle $\mathbb{S}^1$ is strictly bigger than $6$.  Conversely, it is feasible to construct $1$-Lipschitz and non-zero  map $\bar f$: $\bar M^4\to \mathbb{S}^3 \times \mathbb{S}^1 $, while this contradicts to Theorem \ref{comprigidity1}.

To prove Theorem \ref{noncompactrigidity3}, we observe that there is a toric band $\mathbb{T}$ with arbitrarily large width in $\mathbb{R}^{n-3}$. Suppose there is a proper  and $1$-Lipschitz  map  $f:M^n\to \mathbb{S}^3\times\mathbb{R}^{n-3}$ with non-zero degree, one may define a band $\mathbb{L}$ in $M^n$ by choosing a suitable connected component of $f^{-1}(\mathbb{S}^3\times\mathbb{T})$. Note that $\mathbb{L}$ is also a domain in $M^n$, then as what we did in the proof of  Theorem \ref{noncompactrigidity2}, we may get a contradiction by the existence of some $\mu-$bubble in $(M^n, g)$ and Theorem \ref{comprigidity1}. By compactifying the non-compact end of  $M^n$  in a manner analogous to that employed in the proof of the Positive Mass Theorem,  Theorem \ref{noncompactrigidity4} can be reduced to  Theorem \ref{comprigidity1}.

After the completion of this paper, a result akin to Theorem \ref{comprigidity1} which did not necessitate the spin assumption, was identified in the work of  Cecchini, and Schick \cite{CS2021}. They also employed dimension reduction in conjunction with the Dirac operator. However, some of the estimates in their approaches were not sufficiently precise to yield a rigidity result, which is essential in the context of noncompact cases.

The remainder of the paper is outlined as follows. In Section \ref{Sec-minimal-surface}, we present some basic facts of $\mu-$bubble and give the proof Proposition \ref{Prop-rigid}. In Section \ref{Sec-main-result}, we prove our main results.

{\bf Acknowledgement}: The authors wish to express their gratitude to the anonymous referees for identifying numerous typographical errors and for providing suggestions and comments that enhanced the clarity of the paper.

\section{Minimal hypersurface, $\mu$-bubble and Dirac operator}\label{Sec-minimal-surface}
\subsection{Minimal hypersurface and $\mu$-bubble}
In this subsection, we revisit some key facts concerning minimal hypersurfaces and $\mu$-bubbles. For a comprehensive understanding of $\mu$-bubbles, readers are referred to Section 5 in Gromov's paper \cite{Gr2023}.

For a  compact, oriented Riemannian manifold $(M,g)$, if there exists a stable minimal hypersurface $\Sigma$, then the mean curvature of $H_{\Sigma}=0$ in $M$. By the second variation of minimal, oriented hypersurface, the associated Jacobi operator
\[J_{\Sigma}=-\Delta_{\Sigma}-(\Ric_{g}(\nu,\nu)+|A_{\Sigma}|^2)=-\Delta_{\Sigma}+\frac{1}{2}\left(R_{\Sigma}-R_{M}-|A_{\Sigma}|^{2}\right)\]
is nonnegative, where $A_{\Sigma}$ is the second fundamental form of $\Sigma$ and $\nu$ is the unit normal vector of $\Sigma$.

Let $(N^{n},g)$ be a compact Riemannian manifold with boundary. Assume that $\partial N=\partial_{-}N\cup \partial_{+}N$ where $\partial _{\pm}N$ are disjoint closed hypersurface. Fix a smooth function $\mu$ on the interior of $N$, denoted by $\overset{\circ}{N}$ or on $N$. Choose a Caccioppoli set $\Omega_{0}\subset N$ with smooth boundary $\partial \Omega_{0}\subset \overset{\circ}{N}$ and $\partial_{-}N\subset \Omega_{0}$. Consider the following functional
\begin{eqnarray}\label{eqn-functional-times-sphere}
  \mathcal{B}(\Omega)
  =\int_{\partial^{\ast}(\Omega)}d\mathcal{H}^{n-1}- \int_{N}\left(\chi_{\Omega}-\chi_{\Omega_{0}}\right)\mu d\mathcal{H}^{n},
\end{eqnarray}
for $\Omega\in \mathcal{C}$, where $\mathcal{C}$ is defined as
\[\mathcal{C}=\{\Omega|\mbox{all Caccioppoli sets } \Omega \subset N \mbox{ and }\Omega\triangle \Omega_{0}\Subset \overset{\circ}{N}\}, \]
here and in the sequel  $\mathcal{H}^{n}$ denotes $n$-dimensional Hausdorff measure. Then $\mu$-bubbles are critical points of the functional $\mathcal{B}(\Omega)$.

From Zhu's Proposition 2.1 in \cite{Zhu2021}, we have the following existence result of $\mu$-bubble.
\begin{lemma}[Existence of $\mu$-bubble]
For $(N^{n},g)$ as above with $n\leq 7$, if either $\mu\in C^{\infty}(\overset{\circ}{N})$ with $\mu\to \pm\infty$ on $\partial_{\mp}N$, or $\mu\in C^{\infty}(N)$ with
\[\mu|_{\partial_{-}N}>H_{\partial_{-}N},\quad \mu|_{\partial_{+}N}<H_{\partial_{+}N}\]
where $H_{\partial_{-}N}$ is the mean curvature of $\partial_{-}N$ with respect to the inward normal and $H_{\partial_{+}N}$
is the mean curvature of $\partial_{+}N$ with respect to the outward normal. Then there exists an $\Omega\in \mathcal{C}$ with smooth boundary such that
\[\mathcal{B}(\Omega)=\inf_{\Omega'\in\mathcal{C}}\mathcal{B}(\Omega').\]
\end{lemma}

We now discuss the first and second variation of the $\mu$-bubble.
\begin{lemma}[First variation]
  If $\Omega_{\rho}$ is a smooth $1$-parameter family of regions with $\Omega_{0}=\Omega$ and normal speed $\varphi$ at $\rho=0$, then
\begin{eqnarray}\label{eqn-first-variation}
    \frac{d}{d\rho}\mathcal{B}(\Omega_{\rho})     =\int_{\partial \Omega_{\rho}}\left(H_{\partial \Omega}-\mu \right)\varphi d\mathcal{H}^{n-1}
  \end{eqnarray}
  where $H_{\partial \Omega}$ is the mean curvature of $\partial \Omega_{\rho}$ and $\nu$ is the outwards pointing unit normal vector. In particular, a $\mu$-bubble $\Omega$ satisfies
  \[H_{\partial \Omega}=\mu.\]
\end{lemma}

\begin{lemma}[Second variation]
  Consider a $\mu$-bubble $\Omega$ with $\partial\Omega=\mathcal{S}$ ( hence, $H_{\mathcal{S}}=\mu$). Assume that $\Omega_{\rho}$ is a smooth $1$-parameter family of regions with $\Omega_{0}=\Omega$ and normal speed $\varphi$ at $\rho=0$, then
  \begin{eqnarray}\label{eqn-second-variation}
   && \frac{d^2}{d\rho^2}\mathcal{B}(\Omega_{\rho})
   \\ &=& \int_{\mathcal{S}}\left[|\nabla_{\mathcal{S}} \varphi|^2+\frac{1}{2}\left(R_{\mathcal{S}}-R_{N}-H_{\mathcal{S}}^2-|A_{\mathcal{S}}|^2\right)\varphi^2-2\langle\nabla_{N}\mu ,\nu\rangle\varphi^2\right]d\mathcal{H}^{n-1}\nonumber  \end{eqnarray}
\end{lemma}
The Jacobi operator of this second variation is
\begin{eqnarray*}
    \mathcal{J}=-\Delta_{\mathcal{S}}+\frac{1}{2}\left(R_{\mathcal{S}}-R_{N}-H^{2}_{\mathcal{S}}-|A_{\mathcal{S}}|^2-2\langle \nabla_{N}\mu,\nu\rangle\right).
\end{eqnarray*}

\subsection{Dirac operator}
In this subsection, we use the Dirac operator  approach to prove Proposition \ref{Prop-rigid} which is crucial for the proof of our theorems. For foundational insights into the  Dirac operator, readers are referred to the work of Lawson and Michelsohn \cite{Lawson1989}. The proof strategy for Proposition  \ref{Prop-rigid} is inspired by the work of Llarull \cite{LM1998}, and the computational steps share a parallel structure. For the sake of accessibility, we provide the details of the proof (see also  B\"{a}r-Brendle-Hanke-Wang's results in \cite{BBHW2023} for the case with a nonempty boundary.).

\begin{proposition}\label{Prop-rigid}
  Let $(M^{n+3}=\Sigma^{3}\times\mathbb{T}^{n},g)$ be a spin manifold with $g=g_{\Sigma^{3}}+\sum_{i=1}^{n}u_{i}^2d\theta^2$, where $u_{i}$ is positive smooth function on $\Sigma$ for each $1\leq i\leq n$. Assume that for $\mathbb{S}^{3}\times \mathbb{T}^{n}$ with the metric $g_{0}=g_{\mathbb{S}^{3}}+g_{\mathbb{T}^{n}}$, there exists a  map $f=(f_{1},id_{T^{n}}):M=\Sigma\times\mathbb{T}^{n}\to \mathbb{S}^{3}\times \mathbb{T}^{n}$ with non-zero degree and $f_{1}$ is $1$--Lipschitz. If $R_{g}(x)\geq R_{g_{0}}(f(x))=6$ for all $x\in M$. Then  $u_{i}$ are constants for $i=1,\cdots,n$, and   $(M^n, g)$ is isometric to $ \mathbb{S}^{3}\times \mathbb{T}^{1}_{r_{1}}\times\cdots\times \mathbb{T}^{1}_{r_{n}}$, where $r_i$, $i=1,\cdots,n$, are positive constants.
  \end{proposition}

As a corollary, we have

\begin{corollary}[Chai-Pyo-Wan Theorem 1.3 in \cite{chai2023}]\label{Prop-rigid1}
Let $(\Sigma^3, g)$ be a compact orientable Riemannian manifold with the first eigenvalue
$$
\lambda_1(-\Delta +\frac12 R_g)\geq 3,
$$
and there is a $1$-Lipschitz map $f$: $\Sigma^3 \mapsto \mathbb{S}^3$ with non-zero degree,
then $(\Sigma^3, g)$ is isometric to $\mathbb{S}^3$.
\end{corollary}

\begin{proof}[Proof of Proposition \ref{Prop-rigid}]
Since $\mathbb{S}^{3}\times \mathbb{T}^{n}$ has zero Euler characteristic, we need to change it into a manifold with non-zero Euler characteristic. Therefore we consider
\begin{multline}\label{Eqn-map}
  M\times \mathbb{S}^{1}_{r_0}\times \mathbb{S}^{1}_{r_1}\times\cdots\times \mathbb{S}^{1}_{r_{n}}\overset{f\times \frac{1}{r_0}id\times \cdots\times\frac{1}{r_n}id}{\longrightarrow}\mathbb{S}^{3}\times\mathbb{T}^{n}\times\underset{n+1}{\underbrace{\mathbb{S}^{1}\times\cdots\times\mathbb{S}^{1}}}\\
  \overset{h}{\to}(\mathbb{S}^{3}\wedge\mathbb{S}^{1})\times\underset{n}{\underbrace{(\mathbb{T}^{1}\wedge\mathbb{S}^{1})\times\cdots\times (\mathbb{T}^{1}\wedge\mathbb{S}^{1})}}\cong\mathbb{S}^{4}\times \underset{n}{\underbrace{\mathbb{S}^{2}\times\cdots\times\mathbb{S}^{2}}}
\end{multline}
where $\mathbb{S}^{1}_{r_{i}}$ is the one dimensional  circle of radius $r_{i}$ for $i=\{0,1,\cdots,n\}$, $f\times \frac{1}{r_0}id\times \cdots\times\frac{1}{r_n}id$ is defined as
\[
  (f\times \frac{1}{r_0}id\times\cdots \times \frac{1}{r_n}id)(p,t_{0},\cdots,t_{n})=(f(p),\frac{t_{0}}{r_0},\cdots,\frac{t_{n}}{r_n})
\]
for $(p,t_{0},\cdots,t_{n})\in M\times \mathbb{S}^{1}_{r_0}\times\cdots\times \mathbb{S}^{1}_{r_n},$ and where $h$ is a $1$-Lipschitz map into the smash product of non-zero degree see Figure 1.

\begin{center}
\begin{tikzpicture}
\draw (-4,0.1) circle (1);
\draw (-1,0.1) circle (1);
\draw (3,0.1) circle (1);
\draw (6,0.1) circle (1);
\draw [rotate around={180:(-4,0.1)}] (-3,0.1) arc (0:180:1 and .2);
\draw [densely dashed] (-3,0.1) arc (0:180:1 and .2);
\node [above] at (-4,1.1) {$\mathbb{S}^{3}$};
\node [above] at (-1,1.1) {$\mathbb{T}^{1}$};
\node [above] at (3,1.1) {$\mathbb{S}^{1}$};
\node [above] at (6,1.1) {$\mathbb{S}^{1}$};
\node  at (-2.5,0.1) {$\times$};
\node  at (1,0.1) {$\times$};
\node  at (4.5,0.1) {$\times$};
\draw [->] (-4,-1)--(-1.5,-2);
\draw [->] (3,-1)--(-1,-2);
\node  at (-1.25,-1.4) {smash product};
\draw [->] (-1,-1)--(3,-2);
\draw [->] (6,-1)--(3.5,-2);
\node  at (3.25,-1.4) {smash product};
\node  at (-1.25,-2.25) {$\mathbb{S}^{3}\wedge \mathbb{S}^{1}$};
\node  at (3.25,-2.25) {$\mathbb{T}^{1}\wedge \mathbb{S}^{1}$};
\node  at (1,-2) {$\times$};
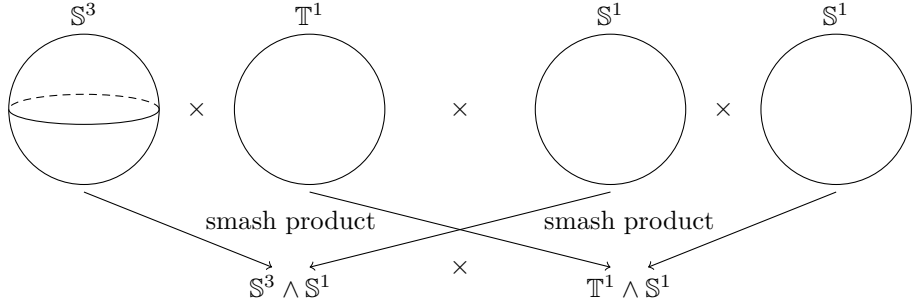
\node  at (1,-2.7) {Figure 1: Example of smash product for $n=1$.};
\end{tikzpicture}
\end{center}

Note that the metric on  $M\times\mathbb{S}_{r_0}^{1}\times \cdots\times \mathbb{S}_{r_n}^{1}$ is $ g+ds^2_{0}+\cdots+ds^2_{n}$ where $ds_{i}^2$ is the standard metric on $\mathbb{S}_{r_i}^{1}$, the metric on $\mathbb{S}^{3}\times\mathbb{T}^{n}\times\mathbb{S}^{1}\times\mathbb{S}^{1}\times\cdots\times\mathbb{S}^{1}$ is $ g_{0}+ds^2+\cdots +ds^2$ where $ds^2$ is the standard metric on $\mathbb{S}^{1}$, and the metric on $\mathbb{S}^{4}\times \mathbb{S}^{2}\times\cdots\times\mathbb{S}^{2}$ is $ \tilde{g}=g_{\mathbb{S}^{4}}+\sum_{i=1}^{n}g_{\mathbb{S}^{2}}$.

Define $\tilde{f}=h\circ(f\times \frac{1}{r_0}id\times\cdots \times \frac{1}{r_n}id)$. Then $\tilde{f}$ is of non-zero degree from $M\times\mathbb{S}_{r_0}^{1}\times\cdots\times \mathbb{S}_{r_n}^{1}$ to  $\mathbb{S}^{4}\times \mathbb{S}^{2}\times\cdots\times \mathbb{S}^{2}$. It is also $1$-Lipschitz for the $\Sigma^{3}$ part, since for vector $v$ tangent to $\Sigma^{3}$,
\begin{eqnarray*}
  \|\tilde{f}_{\ast}(v)\|&=&\|h_{\ast}f_{\ast}(v)\| \leq\|f_{\ast}v\|\\
  &=&\|(f_{1})_{\ast}v\|\leq \|v\|.
\end{eqnarray*}

Let $\hat{M}=M^{n+3}\times \mathbb{S}^1_{r_{0}}\times\cdots\times\mathbb{S}^{1}_{r_{n}}$. Note that the dimension of $\hat{M}$ is $2n+4$, which is even dimension. We adopt Llarull's setting as described in \cite{LM1998}. Consider a complex spinor bundles $S$ over $\hat{M}$ and $E_{0}$ over $\mathbb{S}^{4}\times \mathbb{S}^{2}\times\cdots\times \mathbb{S}^{2}$, respectively. And consider the bundle $S\otimes E$ over $\hat{M}$ for $E=\tilde{f}^{\ast}E_{0}$.

We can choose a basis $\{e_1,e_{2},e_{3},\cdots,e_{n+3},e'_{0}\cdots,e'_{n}\}$ of $g+ds^2_{0}+\cdots+ds^2_{n}$-orthonormal  around $x\in \hat{M}$ such that
$e_1,\cdots,e_{n+3}$ are tangent to $M^{n+3}$ and $e'_{i}$ are tangent to $\mathbb{S}_{r_i}^{1}$ for $i=0,\cdots, n$. Then choose $\tilde{g}$-orthonormal basis
\[\{\epsilon_{1},\epsilon_{2},\epsilon_{3},\cdots,\epsilon_{n+3},\epsilon'_{0},\cdots,\epsilon'_{n}\}\]
around $\tilde{f}(x)$ in $\mathbb{S}^{4}\times \mathbb{S}^{2}\times\cdots\times\mathbb{S}^{2}$ such that $\epsilon_{1},\epsilon_{2},\epsilon_{3},\epsilon'_{0}$ are tangent to $\mathbb{S}^{4}$ and $\epsilon_{3+i},\epsilon'_{i}$ are tangent to $\mathbb{S}^{2}$ for $i=0,\cdots,n$. Moreover, we can require that there exist positive scalars $\{\lambda_{i}\}_{i=1}^{n+3}$ such that $\tilde{f}_{\ast}e_{i}=\lambda_{i}\epsilon_{i}$ and $\{\tau_{i}\}_{i=0}^{n}$ such that $\tilde{f}_{\ast}e'_{i}=\tau_{i}\epsilon'_{i}$. For $\lambda_{i}$ and $1\leq i\leq 3$, since
\[\lambda_{j}^{2}=\tilde{g}(\lambda_{j}\epsilon_{j},\lambda_{j}\epsilon_{j})=\tilde{g}(\tilde{f}_{\ast}e_{j},\tilde{f}_{\ast}e_{j}),\]
then we have
\[\lambda_{j}^{2}=\tilde{g}(\tilde{f}_{\ast}e_{j},\tilde{f}_{\ast}e_{j})\leq g_{0}(f_{\ast}e_{i},f_{\ast}e_{i})\leq g(e_{i},e_{i})=1\]
and $0\leq\lambda_{i}\leq 1$.

For $\tau_{i}$ and $0\leq i\leq n$, we also have
\[\tau_{i}^{2}=\tilde{g}(\tilde{f}_{\ast}e'_{i},\tilde{f}_{\ast}e'_{i})\leq ds^2_{i}(\frac{e'_{i}}{r_{i}},\frac{e'_{i}}{r_{i}})=\frac{1}{r_{i}^2}.\]

For the twisted bundle $S\otimes E$ and its Dirac operator $D_{E}$, by the B-L-W-S formula
\begin{eqnarray}
  D_{E}^2=\nabla^{\ast}\nabla+\frac{R_{g+ds^{2}_{0}+\cdots+ds^{2}_{n}}}{4}+\mathcal{R}^{E}.
\end{eqnarray}
The curvature term $\mathcal{R}^{E}$ can be expressed as(see page 68 in \cite{LM1998}),
\begin{eqnarray*}
  \langle\mathcal{R}^{E}\phi,\phi\rangle &=& \frac{1}{4}\sum_{i\neq j=1}^{3}\sum_{k,l}\sum_{\alpha,\beta}\lambda_{i}\lambda_{j}a_{\alpha\beta}a_{kl}\langle e_{i}\cdot e_{j}\cdot \sigma_{\alpha},\sigma_{k}\rangle\times \langle \epsilon_{j}\cdot \epsilon_{i}\cdot v_{\beta},v_{l}\rangle\\
  && +\frac{1}{4}\sum_{i=1}^{3}\sum_{k,l}\sum_{\alpha,\beta}\lambda_{i}\tau_{0}a_{\alpha\beta}a_{kl}\langle e_{i}\cdot e'_{0}\cdot \sigma_{\alpha},\sigma_{k}\rangle\times \langle \epsilon'_{0}\cdot \epsilon_{i}\cdot v_{\beta},v_{l}\rangle\\
  && +\frac{1}{4}\sum_{j=1}^{3}\sum_{k,l}\sum_{\alpha,\beta}\tau_{0}\lambda_{j}a_{\alpha\beta}a_{kl}\langle e'_{0}\cdot e_{j}\cdot \sigma_{\alpha},\sigma_{k}\rangle\times \langle \epsilon_{j}\cdot \epsilon'_{0}\cdot v_{\beta},v_{l}\rangle\\
  && +\frac{1}{4}\sum_{j=1}^{n}\sum_{k,l}\sum_{\alpha,\beta}\lambda_{3+j}\tau_{j}a_{\alpha\beta}a_{kl}\langle e_{3+j}\cdot e'_{j}\cdot \sigma_{\alpha},\sigma_{k}\rangle\times \langle \epsilon'_{j}\cdot \epsilon_{3+j}\cdot v_{\beta},v_{l}\rangle\\
  &&+\frac{1}{4}\sum_{j=1}^{n}\sum_{k,l}\sum_{\alpha,\beta}\tau_{j}\lambda_{3+j}a_{\alpha\beta}a_{kl}\langle e'_{j}\cdot e_{3+j}\cdot \sigma_{\alpha},\sigma_{k}\rangle\times \langle \epsilon_{3+j}\cdot \epsilon'_{j}\cdot v_{\beta},v_{l}\rangle
\end{eqnarray*}
where $\langle,\rangle=g(,)$.
And, by Lemma 4.5 in \cite{LM1998}, we have
\begin{eqnarray*}
  \langle\mathcal{R}^{E}\phi,\phi\rangle &\geq& -\frac{1}{4}\sum_{i\neq j=1}^{3}\lambda_{i}\lambda_{j}\|\phi\|^2-\frac{1}{2}\left(\sum_{i=1}^{3}\lambda_{i}\tau_{0}\|\phi\|^2\right)-\left(\sum_{i=1}^{n}\frac{1}{2}\lambda_{i+3}\tau_{i}\|\phi\|^2\right)\\
  &\geq&-\frac{1}{4}\sum_{i\neq j=1}^{3}\lambda_{i}\lambda_{j}\|\phi\|^2-\left(\sum_{i=0}^{n}\frac{c_{i}}{r_{i}}\right)\|\phi\|^2
\end{eqnarray*}
for some constant $c_{i}$ independent $r_{i}$.
Hence, we have
\begin{eqnarray*}
  \int_{\hat{M}}\langle D^{2}_{E}\phi,\phi\rangle &\geq&\int_{\hat{M}}\|\nabla\phi\|^2+\int_{\hat{M}}\frac{R_{g}}{4}\langle \phi,\phi\rangle\\
  &&-\frac{1}{4}\sum_{i\neq j=1}^{3}\int_{\hat{M}}\lambda_{i}\lambda_{j}\langle \phi,\phi\rangle-\left(\sum_{i=0}^{n}\frac{c_{i}}{r_{i}}\right)\int_{\hat{M}}\langle \phi,\phi\rangle.
\end{eqnarray*}

Therefore, since $R_{g}\geq 6$, if $R_{g}\neq6$ for some point $p\in M$ or $\lambda_{i}<1$ for some $i=1,2,3$, we can choose $r_{i}$ sufficiency large such that $\langle D^{2}_{E}\phi,\phi\rangle>0$.
 Then $ker(D^2_{E})=0$.  

Now we use $ker(D^2_{E})=0$ to reach a contradiction. This process is almost same with page 65 in \cite{LM1998}.  Recall that $\hat{M}$ is even dimension. The bundle $S$ and $E_{0}$ can be grading as $S=S^{+}\oplus S^{-}$ and $E_{0}=E_{0}^{+}\oplus E_{0}^{-}$ respectively. And $E=E^{+}\oplus E^{-}$ for $E^{\pm}=f^{\ast}E^{\pm}_{0}$. We consider the operator $D_{E^{+}}=D_{E}|_{S\otimes E^{+}}=D_{E}|_{S\otimes f^{\ast}E^{+}_{0}}$.
  Since the operator $D_{E}:\Gamma((S\otimes E^{+})\oplus(S\otimes E^{-}))\to \Gamma((S\otimes E^{+})\oplus(S\otimes E^{-}))$ preserves the direct sum and $ker(D_{E})=ker(D^2_{E})=0$, then
$$0=ker(D_{E|_{S\otimes E^{+}}})=ker(D_{E}).$$
Because $D_{E^{+}}=D^{+}_{E^{+}}\oplus D^{-}_{E^{+}}$, where $D^{\pm}_{E^{+}}:S^{\pm}\otimes E^{+}\to S^{\mp}\otimes E^{+}$ and $0=ker(D_{E^{+}})=ker(D^{+}_{E^{+}})\oplus ker(D^{-}_{E^{+}})$, there is $ker(D^{+}_{E^{+}})=ker(D^{-}_{E^{+}})=0$. By definition, we can write the index of $D_{E^{+}}$ as
\begin{eqnarray}\label{Indexzero}
  Index(D_{E^{+}})=dim(ker(D^{+}_{E^{+}}))-dim(ker(D^{-}_{E^{+}}))=0.
\end{eqnarray}

While, we can prove $Index(D_{E^{+}})\neq 0$. Let $ch(E^{+})$ be the Chern character of $E^{+}$ and $\hat{A}$ be the total $\hat{A}$-class of $\hat{M}$. For the definition of Chern character and $\hat{A}$-class one can see Chapter III section 11 in \cite{Lawson1989}. By the Atiyah-Singer Index Theorem, there is
\[Index(D_{E^{+}})=\{ch(E^{+})\hat{A}(\hat{M})\}[\hat{M}].\]

Let $ch^{i}$ be the $i^{th}$ symmetric polynomial in the Chern class $c_{i}$, with $ch^{i}\in H^{2i}(\hat{M})$.  Then $ch(E^{+})=dim(E^{+})+ch^{1}(E^{+})+\cdots+ch^{n+2}(E^{+})$. For $\mathbb{S}^{4}\times\mathbb{S}^{2}\times\cdots\times \mathbb{S}^{2}$ the Chern  character of the vector bundle $E^{+}_{0}$ is given by
\[ch(E^{+})=dim(E^{+}_{0})+\cdots+\frac{1}{(n+1)!}c_{n+2}(E^{+}_{0}).\]
Then, for $E^{+}=f^{\ast}(E^{+}_{0})$, there is
\begin{eqnarray*}
  ch(E^{+})&=&dim(E^{+}_{0})+\cdots+\frac{1}{(\frac{2n+4}{2}-1)!}\tilde{f}^{\ast}c_{n+2}(E^{+}_{0})\\
  &=&2^{2n+3}+\cdots+\frac{1}{(n+1)!}\tilde{f}^{\ast}c_{n+2}(E^{+}_{0}).
\end{eqnarray*}
Since $\hat{M}=\Sigma^{3}\times \mathbb{T}^{n}\times \mathbb{S}_{r_{0}}^{1}\times\cdots\times  \mathbb{S}_{r_{n}}^{1}$ and $\hat{A}(\Sigma^{3})=1$, $\hat{A}(\mathbb{S}^{1}_{r_{i}})=1$. Then, we have $\hat{A}(\hat{M})=1$. As a result,
\begin{eqnarray*}
  Index(D_{E^{+}}) &=& \{(2^{2n+3}+\cdots+\frac{1}{(n+1)!}\tilde{f}^{\ast}c_{n+2}(E^{+}_{0}))\hat{A}(\hat{M})\}[\hat{M}]\\
  &=&\frac{1}{2}\tilde{f}^{\ast}c_{n+2}(E^{+}_{0})[\hat{M}]=\frac{1}{2}\int_{\hat{M}} \tilde{f}^{\ast}c_{n+2}(E^{+}_{0})\\
  &=&\frac{1}{2}deg(\tilde{f})\int_{\mathbb{S}^{4}\times \mathbb{S}^{2}\times\cdots\times \mathbb{S}^{2}} c_{n+2}(E^{+}_{0}),
\end{eqnarray*}
Since $\chi(\mathbb{S}^{4})$ and $\chi(\mathbb{S}^{2})$ are non-zero, $\chi(\mathbb{S}^{4}\times \mathbb{S}^{2}\times\cdots\times\mathbb{S}^{2})=\chi(\mathbb{S}^{4})\cdot \chi(\mathbb{S}^{2})\cdot\cdots\cdot \chi(\mathbb{S}^{2})\neq 0$. Thus, by $deg(\tilde{f})\neq 0$,
\[\frac{1}{2}deg(\tilde{f})\int_{\mathbb{S}^{4}\times \mathbb{S}^{2}\times\cdots\times \mathbb{S}^{2}} c_{n+2}(E^{+}_{0})\neq 0,\]
because it is non-zero number multiple of the
Euler number of $\mathbb{S}^{4}\times \mathbb{S}^{2}\times\cdots\times\mathbb{S}^{2}$.
This leads to a contradiction with equation (\ref{Indexzero}). Thus $R_{g}\equiv 6$ and $f_{1}$ is isometry. Hence $\Sigma\cong \mathbb{S}^{3}$ and $R_{g_{\Sigma}}=6$. However, by direct computing,
\[R_{g}=R_{g_{\Sigma}}-2\sum_{i=1}^{n}u^{-1}_{i}\Delta_{g_{\Sigma}}u_{i}-2\sum_{1\leq i<j\leq n}\langle \nabla_{g_{\Sigma}}\log u_{i},\nabla_{g_{\Sigma}}\log u_{j}\rangle.\]
Thus, we can deduce that $u_{i}$ must be constant for each $i\in \{1,2,\cdots,n\}$.
\end{proof}

\section{Proof of main results}\label{Sec-main-result}
In this section, we will proceed with the proof of the main theorems in turn. The approach will be sequential, commencing with the compact case and progressing to the noncompact case, followed by a transition from lower to higher codimension. The desired results primarily encompass two aspects: nonexistence and rigidity, which will be individually discussed in various cases.
\subsection{Proof of Theorem \ref{comprigidity1}}
In this subsection, we focus on the compact case, in which we can establish the rigidity result for such manifolds. We prove Theorem \ref{comprigidity1} , specifically addressing the case $n=4$. For the case $5\leq n\leq 7$, the argument follows analogously due to the dimension reduction trick.

The proof of Theorem \ref{comprigidity1} begins with identifying a minimal hypersurface $\Sigma^3$ in $(M^4, g)$, followed by verifying that $\Sigma^3$ is isometric to the standard unit sphere $\mathbb{S}^3$. Subsequently, we construct a foliation near $\Sigma^3$, with each slice being isometric to the standard unit sphere $\mathbb{S}^3$. This establishes the local isometry of $(M^4, g)$ to $\mathbb{S}^3\times \mathbb{S}^{1}$.

Let us beginning with the following lemma, which enables us to construct a foliation of hypersurfaces with constant mean curvature in $M^{4}$ near $\Sigma$. 
\begin{lemma}[Bray-Brendle-Neves Proposition 3.2 \cite{HSA2010} or Zhu Lemma 3.3 \cite{Zhu2020}]\label{Lem-foliation-construct}
  If $\Sigma$ is an area-minimizing hypersurface in $M^{4}$ with vanished second fundamental form and vanished normal Ricci curvature, then we can construct a local foliation $\{\Sigma_{t}\}_{-\epsilon\leq t\leq \epsilon}$ in $M^{4}$ such that $\Sigma_{t}$ are of constant mean curvatures and $\Sigma_{0}=\Sigma$.
\end{lemma}

Now, let us recall Theorem \ref{comprigidity1} specifically for $n=4$.

\begin{theorem}
Let $(M^4, g)$ be a $4$-dimensional compact connected orientable Riemannian manifold without boundary, satisfying $R_g\geq 6$. Assume there is a non-zero degree and $1$-Lipschitz map $f:M^4\to \mathbb{S}^3\times \mathbb{S}^{1}$. Then $(M^4, g)$ is locally isometric to $\mathbb{S}^3\times \mathbb{S}^{1}$. Furthermore, for any $x\in\mathbb{S}^{1} $, $P \circ f(\cdot,x): \mathbb{S}^3 \to \mathbb{S}^3 $ is isometric.
\end{theorem}

\begin{proof}
Since the map $f$ has non-zero degree, according to the geometric measure theory(see Section 5.1.6 in Federer's book \cite{FH1969}), there is a closed orientable area-minimizing hypersurface $\Sigma^3$ of $M^4$ in the homologous class  $[f^{-1}(\mathbb{S}^{3}\times\{p\})]\in H_3(M, Z)$. So $f|_{\Sigma}:\Sigma^3\to \mathbb{S}^3$ has non-zero degree. Then by the stability condition, we obtain $J_{\Sigma}\geq 0$. Let $u>0$ denote the first eigenfunction of $J_{\Sigma}$. Then, 

\begin{equation}\label{Eqn-Jacobi}
\begin{split}
J_{\Sigma}u&=-\Delta_{\Sigma}u-(\Ric_{g}(\nu,\nu)+|A_{\Sigma}|^2)u\\
&=-\Delta_{\Sigma}u+\frac{1}{2}\left(R_{\Sigma}-R_{g}-|A_{\Sigma}|^{2}\right)u\\
&=\lambda u
	\end{split}
\end{equation}
where $\lambda\geq 0$ is the first eigenvalue. Therefore, if we consider the manifold $(\Sigma\times\mathbb{S}^{1}, g_{u}=g|_{\Sigma}+u^2dt^2)$, which is spin, the scalar curvature
\[R_{g_{u}}=\frac{-2\Delta_{\Sigma}u}{u}+R_{\Sigma}=R_{g}+|A_{\Sigma}|^2+2\lambda\geq 6.\]
Since $( f_{1}=P_{1}\circ f|_{\Sigma},id):\Sigma\times \mathbb{S}^{1}\to \mathbb{S}^{3}\times\mathbb{S}^{1}$ is a non-zero degree map with $f_{1}$ being 1-Lipschitz, where $P_{1}:\mathbb{S}^{3}\times\mathbb{S}^{1}\to \mathbb{S}^{3}$ is the canonical projection, according to Proposition \ref{Prop-rigid}, we deduce that $\Sigma$ is the standard  unit $\mathbb{S}^{3}$ and $u$ is constant. Utilizing the equation (\ref{Eqn-Jacobi}), we obtain $\lambda=0$, $A_{\Sigma}=0$ and $\Ric_{g}(\nu,\nu)=0$.

Now we establish the existence of a local isometry $\tilde{\Phi}$ as stated in Theorem \ref{comprigidity1}.
According to Lemma \ref{Lem-foliation-construct}, there exists a constant mean curvature foliation $\{\Sigma_{t}\}_{-\epsilon\leq t\leq \epsilon}$ in $M$ with $\Sigma_{0}=\Sigma$. Let $H_{t}$ denote the mean curvature of $\Sigma_{t}$.

We claim that $H_{t}\equiv0$ for any $t\in (-\epsilon,\epsilon)$. It suffices to prove this for non-negative $t$, as the argument for the negative side is analogous. For contradiction, assume that the assertion fails. Then there exists a $t_{0}>0$ such that $H_{t_{0}}>0$, since otherwise $H_{t}\leq 0$ and the area-minimizing property of $\Sigma_{t}$ would imply $H_{t}\equiv0$. Let $\Omega_{t_{0}}$ denote the region enclosed by $\Sigma_{t_{0}}$ and $\Sigma$, and define the brane functional

\begin{eqnarray}
  \mathcal{B}(\Omega) = \int_{\partial \Omega\backslash \Sigma} d\mathcal{H}^{3}-\delta \int_\Omega d\mathcal{H}^{4},\quad 0<\delta<H_{t_{0}},
\end{eqnarray}
where $\Omega$ is any Borel subset of $\Omega_{t_{0}}$ with finite perimeter and $\Sigma\subset \partial \Omega$. Due to $0<\delta<H_{t_{0}}$, the hypersurfaces $\Sigma$ and $\Sigma_{t_{0}}$ serve as barriers. Therefore we can find a Borel set $\hat{\Omega}_{t_{0}}$ minimizing the brane functional $\mathcal{B}$, for which $\hat{\Sigma}=\partial \hat{\Omega}_{t_{0}}\backslash \Sigma$ is a smooth 2-side hypersurface disjoint from $\Sigma$ and $\Sigma_{t_{0}}$. Denote the induced metric of $\hat{\Sigma}$ from $(M,g)$ by $\hat{g}$. Since $\hat{\Sigma}$ is $\mathcal{B}$-minimizing, it has constant mean curvature $\delta$ and it is $\mathcal{B}$-stable. So the Jacobi operator
\begin{eqnarray}
  \hat{\mathcal{J}}=-\Delta_{\hat{g}}-(\operatorname{Ric}_{g}(\hat{\nu},\hat{\nu})+|\hat{A}|^2)=-\Delta_{\hat{g}}-\frac{1}{2}(R_{g}-R_{\hat{g}}+\delta^2+|\hat{A}|^2)
\end{eqnarray}
is nonnegative, where $\hat{\nu}$ is the unit normal vector field of $\hat{\Sigma}$ and $\hat{A}$ is the corresponding second fundamental form. Taking the first eigenfunction $\hat{u}$ of $\hat{\mathcal{J}}$ with respect to the first eigenvalue $\hat{\lambda}_{1}\geq 0$ and defining the metric
\begin{eqnarray}
  \bar{g}=\hat{g}+\hat{u}^2dt^2
\end{eqnarray}
on $\hat{\Sigma}\times\mathbb{S}^{1}$, we obtain
\[R_{\bar{g}}= R_{g}+\delta^2+|\hat{A}|^2+2\hat{\lambda}_{1}\geq 6+\delta^2.\]
Since $\hat{\Sigma}$ is homologous to $\Sigma$, the map $\hat{f}=(f_{1}=P_{1}\circ f|_{\hat{\Sigma}},id):\hat{\Sigma}\times\mathbb{S}^{1}\to \mathbb{S}^{3}\times\mathbb{S}^{1}$ has non-zero degree and $f_{1}$ is $1$-Lipschitz. By Proposition \ref{Prop-rigid}, we have $|\hat{A}|=0$ and $\delta=0$, which leads to a contradiction.

Since $H_{t}$ vanishes for any $t\in (-\epsilon,\epsilon)$, there holds $A(\Sigma_{t})=A(\Sigma)$, which yields that $\Sigma_{t}$ is also area-minimizing. Replacing $\Sigma$ by $\Sigma_{t}$ and combining with Proposition \ref{Prop-rigid}, $\Sigma_{t}$ is totally geodesic and has vanished normal Ricci curvature.

In the following, we establish the existence of a local isometry $\tilde{\Phi}:\Sigma\times \mathbb{R}\to M$. Let $\tilde{V}$ denote the normal variation vector field of the foliation $\Sigma_{t}$ and let $\tilde{\Phi}:\Sigma\times (-\epsilon,\epsilon)\to M$ be the flow generated by $\tilde{V}$. It follows that $\tilde{\Phi}$ is an embedding in a small neighborhood of $\Sigma$ and the pull-back of the metric $g$ is
  \[\tilde{\Phi}^{\ast}(g)=\phi^2dt^2+\tilde{\Phi}^{\ast}_{t}(g_{t}),\]
  where $\phi>0$ is the lapse function (see definition in (3.3) of \cite{HSA2010}) and $g_{t}$ is the induced metric on $\Sigma_{t}$ from $g$. Since $\Sigma_{t}$ is totally geodesic, we have $\partial_{t}\tilde{\Phi}^{\ast}_{t}(g_t)=2\phi A_{\Sigma_{t}}=0$, which implies that $\tilde{\Phi}_{t}^{\ast}(g_t)\equiv g|_{\Sigma}$. Additionally, the stability condition yields $\tilde{J}\phi=-\Delta\phi=0$, hence $\phi(\cdot,t)\equiv \text{const}$. Let
  \[r(t)=\int_{0}^{t}\phi(\cdot,s)ds.\]
  Then $\tilde{\Phi}^{\ast}(g)=dr^2+g|_{\Sigma}$. Therefore, $\tilde{\Phi}:\Sigma\times (-\epsilon,\epsilon)\to M$ is a local isometry. Through a continuous argument as Proposition 3.8 in \cite{HSA2010}, we conclude that there exists a local isometry $\tilde{\Phi}:\Sigma\times\mathbb{R}\to M$.
\end{proof}

\begin{remark}
    For the case $5\leq n\leq 7$, as in \cite{Zhu2020}, we need to make multiple $\mathbb{S}^{1}$ symmetry and the process is the  same as that in \cite{Zhu2020}. We only need to check  the $1$-Lipschitz property. However this holds when we restrict to each $\Sigma_{k}$ as construction in Proposition 2.2 in \cite{Zhu2020}.
\end{remark}
\subsection{Proof of Theorem \ref{noncompactrigidity1}}
In this subsection, we focus on the noncompact case with codimension 1 and prove Theorem \ref{noncompactrigidity1}. The key step is to find a minimal hypersurface in noncompact $M^{4}$. Then the remaining proof is similar to that of Theorem \ref{comprigidity1}.

\begin{theorem}
Let $(M^4, g)$ be a noncompact orientable connected, complete Riemannian manifold with scalar curvature $R_g\geq 6 $ and $f:M^4\to \mathbb{S}^3\times\mathbb{R}$ be a proper, $1$-Lipschitz  map with non-zero degree.	 Then $(M^4, g)$ is isometric to $\mathbb{S}^3\times\mathbb{R}$ provided $(M^4, g)$ has bounded geometry, i.e. $\sup_{M}\|Rm\|<\infty$ and its injective radius $inj(M^4, g)>0$.
\end{theorem}

Before the proof, we need some functions constructed by Zhu in \cite{ZJT2020}.

\begin{lemma}[Lemma 2.1 in \cite{ZJT2020}]\label{Lem-function-phi}
  There is a proper and surjective smooth function $\phi:M\to \mathbb{R}$ such that $\phi^{-1}(0)=\Sigma$ and $\operatorname{Lip}\phi<1$.
  Where $\Sigma\subset M$  is an orientable closed hypersurface associated with a signed distance function.
\end{lemma}

\begin{lemma}[Lemma 2.3 in \cite{ZJT2020}]\label{Lem-function-h}
  For any $\epsilon\in (0,1)$, there is a function
  \[h_{\epsilon}:\left(-\frac{1}{4\epsilon},\frac{1}{4\epsilon}\right)\to \R\]
  such that
  \begin{enumerate}[(1)]
    \item $h_{\epsilon}$ satisfies
    \[\frac{4}{3}h^2_{\epsilon}+2h_{\epsilon}^{'}=12\epsilon^2\quad\text{on}\quad \left(-\frac{1}{4\epsilon},-\frac{1}{8}\right]\cup\left[\frac{1}{8},\frac{1}{4\epsilon}\right)\]
    and there is a universal constant $C$ so that
    \[\sup_{-\frac{1}{2n}\leq t\leq \frac{1}{2n}}\left|\frac{4}{3}h_{\epsilon}^2+2h_{\epsilon}^{'}\right|\leq C\epsilon.\]
    \item $h^{'}_{\epsilon}<0$ and
    \[\lim_{t\to\mp \frac{1}{4\epsilon}}h_{\epsilon}(t)=\pm\infty.\]
    \item As $\epsilon\to 0$, $h_{\epsilon}$ converge smoothly to $0$ on any closed interval.
    \item $h_{\epsilon}<0$ on $\left[\frac{1}{8},\frac{1}{4\epsilon}\right)$.
  \end{enumerate}
\end{lemma}

For the function $\phi$ constructed in Lemma \ref{Lem-function-phi}, denote
\[\Omega_{0}=\{x\in M|-\frac1{\epsilon}<\phi(x)<0\}.\]
Given any smooth function $h:(-T,T)\to\R$, define the functional, i.e. $\mu$-bubble, as
\[\mathcal{B}^{h}(\Omega)=\mathcal{H}^{3}(\partial^{\ast} \Omega)-\int_{M}(\chi_{\Omega}-\chi_{\Omega_{0}})h\circ \phi d\mathcal{H}^{4},\]
on
\[\mathcal{C}_{T}=\{\mbox{Caccipoli set }\Omega\subset M|\Omega \triangle \Omega_{0}\Subset \phi^{-1}((-T,T))\}.\]
For the minimization problem of the functional $\mathcal{B}^{h}$ on $\mathcal{C}_{T}$, we have the following existence result.

\begin{lemma}[Lemma 2.2 in \cite{ZJT2020}]\label{Lem-existence}
  Assume that $\pm T$ are regular values of $\phi$ and also the function $h$ satisfies
  \begin{align}\label{Eqn-barrier-condition}
    \lim_{t\to-T}h(t)=+\infty\quad\mbox{and} \quad \lim_{t\to T}h(t)=-\infty,
  \end{align}
  then there exists a smooth minimizer $\widehat{\Omega}$ in $\mathcal{C}_{T}$ for $\mathcal{B}^{h}$.
\end{lemma}

Because $f$ is a non-zero degree map, there exists an orientable closed hypersurface $\Sigma$ satisfying $(P_{1}\circ f)_{\ast}[\Sigma]=[\mathbb{S}^{3}]$, where $[\Sigma]$ is not zero element in homology group and $P_{1}\circ f|_{\Sigma}:\Sigma\to \mathbb{S}^{3}$ is a non-zero degree map.
For this hypersurface $\Sigma$, we can use invoke \ref{Lem-function-phi} to construct a function $\phi$ such that $\Sigma=\phi^{-1}(0)$.

Now considering the function $h_{\epsilon}$ constructed in Lemma \ref{Lem-function-h}, we define $\mu_{\epsilon}=h_{\epsilon}\circ \phi$. Let functional 
\[\mathcal{B}^{\epsilon}(\Omega)=\mathcal{H}^{3}(\partial^{\ast} \Omega)-\int_{M}(\chi_{\Omega}-\chi_{\Omega_{0}})\mu_{\epsilon} d\mathcal{H}^{4},\]
be defined on
\[\mathcal{C}_{\epsilon}=\{\mbox{Caccipoli set }\Omega\subset M|\Omega \triangle \Omega_{0}\Subset \phi^{-1}\left(\left(-\frac{1}{4\epsilon},\frac{1}{4\epsilon}\right)\right)\}.\]
By combining Lemma \ref{Lem-existence} with the Sard's theorem, we conclude (as stated in Proposition 2.4 in \cite{ZJT2020}) that for almost every $\epsilon\in (0,1)$, there exists a smooth minimizer $\widehat{\Omega}_{\epsilon}$ in $\mathcal{C}_{\epsilon}$ for the functional $\mathcal{B}^{\epsilon}$.

Let $\Sigma_{\epsilon}=\partial \widehat{\Omega}_{\epsilon}$. Our objective is to prove the convergence of $\Sigma_{\epsilon}$ to a closed minimal hypersurface. The following lemma asserts that the volume of $\Sigma_{\epsilon}$ is uniformly bounded.

\begin{lemma}
   Let $\epsilon\in (0,1)$, then we have $\mathcal{H}^{3}(\Sigma_{\epsilon})\leq \Lambda<\infty$ for some constant $\Lambda$  which is independent on $\epsilon$.
\end{lemma}

\begin{proof}
Because $\widehat{\Omega}_{\epsilon}$ is a minimizer,
  \begin{equation} 
  \begin{split}
  	\mathcal{B}^{\epsilon}(\widehat{\Omega}_{\epsilon}) & =\mathcal{H}^{3}(\partial \widehat{\Omega}_{\epsilon})-\int_{\Omega_{0}}(\chi_{\widehat{\Omega}_{\epsilon}}-\chi_{\Omega_{0}})\mu_{\epsilon}d\mathcal{H}^{4}-\int_{M\setminus\Omega_{0}}(\chi_{\widehat{\Omega}_{\epsilon}}-\chi_{\Omega_{0}})\mu_{\epsilon}d\mathcal{H}^{4}\\
    &=\mathcal{H}^{3}(\partial \widehat{\Omega}_\epsilon)-\int_{\{0\leq\phi\leq \frac{1}{8}\}}\chi_{\widehat{\Omega}_{\epsilon}}\mu_{\epsilon}d\mathcal{H}^{4}-\int_{\{\phi>\frac{1}{8}\}}\chi_{\widehat{\Omega}_{\epsilon}}\mu_{\epsilon}d\mathcal{H}^{4}\\
    &\leq \mathcal{H}^{3}(\Sigma).
    \end{split}
    \end{equation}
  Let
  \[\Lambda=\left|\int_{\{0\leq\phi\leq \frac{1}{8}\}}\chi_{\Omega_{\epsilon}}\mu_{\epsilon}d\mathcal{H}^{4}\right|+\mathcal{H}^{3}(\Sigma).\]
  Since $\mu_{\epsilon}<0$ on $\left[\frac{1}{8},\frac{1}{4\epsilon}\right)$, we have
  \[\mathcal{H}^{3}(\Sigma_{\epsilon})=\mathcal{H}^{3}(\partial \widehat{\Omega}_{\epsilon})\leq \Lambda.\]
\end{proof}

By geometric measure theory (see Lemma 4.10 in \cite{HLS2023}), for any fixed compact set $K\subset M$, it follows that $\Sigma_{\epsilon}\cap K\to \mathcal{S}\cap K$ as $\epsilon\to 0$. Recalling property $(3)$ of $h_{\epsilon}$, we deduce that $\mathcal{S}$ is a stable minimal surface in the compact set $K$. Subsequently, we establish the compactness of $\mathcal{S}$.

\begin{lemma}\label{Lem-compact-surface}
  $\mathcal{S}\subset M$ is compact.
\end{lemma}

\begin{proof}
Since $\mathcal{S}$ is minimal and $M$ has bounded geometry, by Lemma 1 in \cite{MY1980}, for any $r_0<inj(M^n,g)$ and any $p\in \mathcal{S}$, we have  $\mathcal{H}^{3}(B_{r_{0}}(p)\cap \mathcal{S})\geq C_{r_0}>0$ for some constant $C_{r_0}$ depending only on $r_{0}$, $inj(M^4,g)$ and $\sup_M \|Rm||$, where $B_{r_{0}}(p)$ denotes the geodesic ball in $(M^4,g)$ with radius $r_0$ and center $p$. Then for $\epsilon$ sufficiency small, $\mathcal{H}^{3}(B_{r_{0}}(p)\cap\Sigma_{\epsilon})\geq \frac12 C_{r_0}>0$. Therefore, by a suitable covering argument, we know that if $\mathcal{S}$ is noncompact, then $\mathcal{H}^{3}(\mathcal{S})=\infty$, which implies  $\mathcal{H}^{3}(\Sigma_{\epsilon})\to\infty$ as $\epsilon\to 0$. While $\mathcal{H}^{3}(\Sigma_{\epsilon})\leq \Lambda<\infty$, we get a contradiction.
\end{proof}

\begin{proof}[Proof of Theorem \ref{noncompactrigidity1}] By Lemma \ref{Lem-compact-surface}, there exists a closed oriented hypersurface $\mathcal{S}$ which is homologous to $\Sigma$. Consequently, the map $P_{1}\circ f|_{\mathcal{S}}:\mathcal{S}\to \mathbb{S}^{3}$ has non-zero degree and is also $1$-Lipschitz. Therefore, the map $(P_{1}\circ f|_{\mathcal{S}},id):\mathcal{S}\times\mathbb{S}^{1}\to \mathbb{S}^{3}\times\mathbb{S}^{1}$ also has non-zero degree. The remaining proof is the same as that of Theorem \ref{comprigidity1}, concluding that $(M^4,g )$ is locally isometric to $\mathbb{S}^3\times \mathbb{R}$. Thus $(M^{4},g)$ has nonnegative Ricci curvature. Utilizing the splitting Theorem and noting that  $M^{4}$ has at least two ends, we deduce that $(M^{4},g)$ is isometric to $\mathbb{S}^{3}\times \mathbb{R}$.
\end{proof}
\subsection{Proof of Theorem \ref{noncompactrigidity2}}
In this subsection, we focus on the nonexistence and prove Theorem \ref{noncompactrigidity2}. The bounded geometry assumption becomes dispensable in the absence of a rigidity conclusion. A crucial step  is to construct the function $\mu$.
\begin{theorem}
Let $(M^4, g)$ be a  noncompact orientable connected complete Riemannian manifold with scalar curvature $R_g> 6$. Then there is no proper and $1$-Lipschitz  map  $f:M^4\to \mathbb{S}^3\times\mathbb{R}$ with non-zero degree.
\end{theorem}
\begin{proof}
We prove by contradiction. Suppose there exists a such map $f$. Since $f$ is a non-zero degree map, there exists an orientable closed hypersurface $\Sigma$ such that $(P_{1}\circ f)_{\ast}[\Sigma]=[\mathbb{S}^{3}]$, where $[\Sigma]$ is non-zero in homology group and $P_{1}\circ f|_{\Sigma}:\Sigma\to \mathbb{S}^{3}$ is a non-zero degree map.
For this $\Sigma$, we can use Lemma \ref{Lem-function-phi} to construct $\phi$ satisfying $\Sigma=\phi^{-1}(0)$.

Firstly, we fix a large compact set $K=\phi^{-1}([-T,T])$ such that $R_{g}\geq 6+\delta_{K}$ and $\frac{\delta_{K}}{2}>12\kappa $ for some $T>0$ and small $\kappa>0$. Then we construct a function $h_{\kappa}$ as the function in the process of proof of Corollary 3.6 in \cite{Ce_2023}. For some positive constants $\epsilon_{1}$, $\epsilon_{2}<\frac{\pi}{4\sqrt{\kappa}}$, $\epsilon_{3}$, $\delta_{1}$, $\delta_{2}$ and $a,b>0$, let
    \begin{eqnarray}
        h_{\kappa}&=\begin{cases}
            h_{1}=\frac{3}{2(t+a)}, \mbox{ if } t\in [-a+\delta_{1}, -\epsilon_{1}] \\
            h_{2}=-3\sqrt{\kappa}\tan\left({2\sqrt{\kappa}t}\right), \mbox{ if } t\in \left[-\epsilon_{2},\epsilon_{2}\right] \\
            h_{3}=-\frac{3}{2(b-t)} \mbox{ if } t\in [\epsilon_{3},b-\delta_{2}].
        \end{cases}
    \end{eqnarray}

    Since $h_{1}(t)$ can be arbitrarily large  when $t$ is sufficiently close to  $-a$,  and can be arbitrarily close to $\frac{3}{2a}$ when $|t|$ is small enough, we can choose $a>T>\epsilon_{1}>\epsilon_{2}>0$ such that $h_{1}(-\epsilon_{1})>h_{2}(-\epsilon_{2})$ and choose $\delta_{1}>0$ such that $h_{1}(-a+\delta_{1})> H_{\{(h_{1}\circ\phi)^{-1}(-a+\delta_{1})\}}$ and $a-\delta_1+\epsilon>T$ for some small $\epsilon>0$,  where $H_{\{(h_{1}\circ\phi)^{-1}(t)\}}$ denotes the mean curvature of the hypersurface $(h_{1}\circ\phi)^{-1}(t)$  with respect to the outward unit normal vector.

    Similarly, we can choose  $0<\epsilon_{2}<\epsilon_{3}<T<b$ such that $h_{3}(\epsilon_{3})>h_{2}(\epsilon_{2})$ and choose $\delta_{2}>0$ such that $h_{3}(b-\delta_{2})< H_{\{(h_{3}\circ\phi)^{-1}(b-\delta_{2})\}}$ and $b-\delta_{2}+\epsilon>T$. Then we can use the function $\hat{h}$ constructed in Lemma 6.1 in \cite{Ce_2023} to smooth $h_{\kappa}$  and then get
    $\mu_{\kappa}$ satisfying
    \[\mu_{\kappa}(-a+\delta_{1}-\epsilon)>H_{\{(\mu\circ \phi)^{-1}(-a+\delta_{1}-\epsilon)\}},\quad \mu_{\kappa}(b-\delta_{2}+\epsilon)<H_{\{(\mu\circ \phi)^{-1}(b-\delta_{2}+\epsilon)\}}\]
    and
    \begin{eqnarray}
    \begin{cases}
        -\frac{4}{3}\mu^2_{\kappa}(t)-2\mu_{\kappa}^{'}(t)\leq 12\kappa ,\quad \mbox{ if } t\in (-T,T) ;\\
        -\frac{4}{3}\mu^2_{\kappa}(t)-2\mu_{\kappa}^{'}(t)\leq 0,\quad \mbox{ if } t\in [-a+\delta_{1}-\epsilon,-T]\cup [T,b-\delta_{2}+\epsilon];\\
      \mu_{\kappa}^{'}\leq 0 , \quad \mbox{ if } t\in [-a+\delta_{1}-\epsilon,b-\delta_{2}+\epsilon].
    \end{cases}
    \end{eqnarray}

Now we consider the functional
\[\mathcal{B}^{\kappa}(\Omega)=\mathcal{H}^{3}(\partial^{\ast} \Omega)-\int_{M}(\chi_{\Omega}-\chi_{\Omega_{0}})\mu_{\kappa}\circ \phi d\mathcal{H}^{4},\]
defined on
\[\mathcal{C}_{\kappa}=\{\mbox{Caccipoli set }\Omega\subset M|\Omega \triangle \Omega_{0}\Subset \phi^{-1}\left(\left(-a+\delta_{1}-\epsilon,b-\delta_{2}+\epsilon\right)\right)\}.\]
   According to Lemma \ref{Lem-existence}, there is a smooth minimizer $\bar{\Omega}$ in $\mathcal{C}_{\kappa}$ for the functional $\mathcal{B}^{\kappa}$.
   Let $\Bar{\Sigma}=\partial \Bar{\Omega}$, then $\Bar{\Sigma}$ is homologous to $\Sigma$. Therefore the map $f_{1}=P_{1}\circ f_{\Bar{\Sigma}}:\Bar{\Sigma}\to \mathbb{S}^{3}$ has non-zero degree and is $1$-Lipschitz. By the stability condition of the $\mu_{\kappa}$-bubble, we have
   \begin{eqnarray*}
     -\Delta_{\Bar{\Sigma}}u+\frac{1}{2}R_{\Bar{\Sigma}}u&\geq& \frac{1}{2}( R_{g}u+\left(\frac{4}{3}\mu_{\kappa}^2-2|\mu'_{\kappa}||\nabla \phi|\right)u)\\
     &\geq& \frac{1}{2}(R_{g}u+\left(\frac{4}{3}\mu_{\kappa}^2+2\mu'_{\kappa}\right)u)\\
     &>& 3 u.
   \end{eqnarray*}
   Thus, on $(\Bar{\Sigma}\times \mathbb{S}^{1})$ endowed with the metric $\bar{g}=g_{\Bar{\Sigma}}+u^2d\theta^2$, we have $R_{\Bar{g}}>6$ and the map $(f_{1},id):\Bar{\Sigma}\times\mathbb{T}^{1}\to \mathbb{S}^{3}\times \mathbb{S}^{1}$ has non-zero degree, with $f_{1}$ being $1$-Lipschitz. Combining with Proposition \ref{Prop-rigid}, we get a contradiction. This implies that there is no such map $f$.
 \end{proof}
\subsection{Proof of Theorem \ref{noncompactrigidity3}}
In this subsection, we focus on the noncompact case with higher codimension and prove Theorem \ref{noncompactrigidity3}. Since the Example \ref{example1} provides a counterexample to both the rigidity and the nonexistence, we only consider the nonexistence of manifolds with larger uniformly lower bound of scalar curvature here.

\begin{theorem}
For any $\delta>0$, let $(M^n, g)$, where $5\leq n\leq 7$, be a  noncompact orientable connected complete Riemannian manifold with scalar curvature $R_g\geq 6+\delta $.  Then there is no proper and $1$-Lipschitz  map  $f:M^n\to \mathbb{S}^3\times\mathbb{R}^{n-3}$ with non-zero degree.
\end{theorem}

\begin{proof}
	Let $\mathbb{T}$ be a torical band in $\mathbb{R}^{n-3}$ with sufficient large width. So there is a proper and $1$-Lipschitz map $\Phi$: $\mathbb{S}^3 \times\mathbb{T} \to \mathbb{S}^3 \times \mathbb{T}^{n-4}\times [-1,1]$ with non-zero degree. Let $\mathbb{L}$ be one of components of $f^{-1}( \mathbb{S}^3 \times\mathbb{T})$	such that $\partial \mathbb{L}:=\partial_+\mathbb{L} \cup \partial_-\mathbb{L}$ are smooth hypersurfaces, and $f: \mathbb{L}\to \mathbb{S}^3 \times\mathbb{T} $ has non-zero degree. Here we assume
	
	$$
	 \Phi\circ f(\partial_{\pm} \mathbb{L})=\mathbb{S}^3 \times \mathbb{T}^{n-4}\times \{\pm 1\}	,
	$$
and let
 $$P:\mathbb{S}^3 \times \mathbb{T}^{n-4}\times [-1,1]\to \mathbb{S}^3 \times \mathbb{T}^{n-4}\times \{-1\} ,
 $$	
be the standard projection (As Figure 2 shows.).
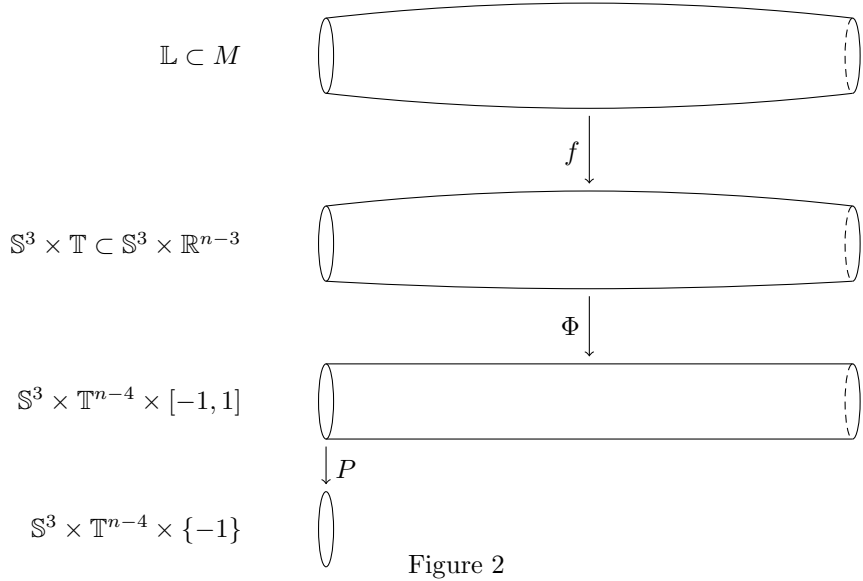
\begin{center}
\begin{tikzpicture}
\draw (2,2.8) parabola bend (5.5,3) (9,2.8);
\draw (2,1.8) parabola bend (5.5,1.6) (9,1.8);
\draw (2,2.8) arc (90:270:.1 and .5);
\draw [rotate around={180:(2,2.3)}] (2,2.8) arc (90:270:.1 and .5);
\draw [densely dashed] (9,2.8) arc (90:270:.1 and .5);
\draw [rotate around={180:(9,2.3)}] (9,2.8) arc (90:270:.1 and .5);
\node[left] (L) at (1,2.3) {$\mathbb{L}\subset M$};
\draw (2,0.3) parabola bend (5.5,0.5) (9,0.3);
\draw (2,-0.7) parabola bend (5.5,-0.8) (9,-0.7);
\draw (2,0.3) arc (90:270:.1 and .5);
\draw [rotate around={180:(2,-0.2)}] (2,0.3) arc (90:270:.1 and .5);
\draw [densely dashed] (9,0.3) arc (90:270:.1 and .5);
\draw [rotate around={180:(9,-0.2)}] (9,0.3) arc (90:270:.1 and .5);
\node[left] (L) at (1,-0.2) {$\mathbb{S}^{3}\times \mathbb{T}\subset \mathbb{S}^{3}\times\mathbb{R}^{n-3}$};
\draw (2,-1.8) parabola bend (5.5,-1.8) (9,-1.8);
\draw (2,-2.8) parabola bend (5.5,-2.8) (9,-2.8);
\draw (2,-1.8) arc (90:270:.1 and .5);
\draw [rotate around={180:(2,-2.3)}] (2,-1.8) arc (90:270:.1 and .5);
\draw [densely dashed] (9,-1.8) arc (90:270:.1 and .5);
\draw [rotate around={180:(9,-2.3)}] (9,-1.8) arc (90:270:.1 and .5);
\node[left] (L) at (1,-2.3) {$\mathbb{S}^{3}\times \mathbb{T}^{n-4}\times [-1,1]$};
\draw (2,-3.5) arc (90:270:.1 and .5);
\draw [rotate around={180:(2,-4)}] (2,-3.5) arc (90:270:.1 and .5);
\node[left] (L) at (1,-4) {$\mathbb{S}^{3}\times \mathbb{T}^{n-4}\times \{-1\}$};
\draw [->] (5.5,1.5)--(5.5,0.6);
\node[left] at (5.5,1) {$f$};
\draw [->] (5.5,-0.9)--(5.5,-1.7);
\node[left] at (5.5,-1.3) {$\Phi$};
\draw [->] (2,-2.9)--(2,-3.4);
\node[right] at (2,-3.2) {$P$};
\node[left] at (4.5,-4.5) {Figure 2};
\end{tikzpicture}
\end{center}

 It is straightforward to observe that $\Phi\circ f:\mathbb{L}\to \mathbb{S}^3 \times \mathbb{T}^{n-4}\times [-1,1]$ is a proper and $1$-Lipschitz  map  with non-zero degree.
Let $a=\sqrt{\frac{\delta}{2n(n-1)}}$, and define
$$
\mu(\rho)=(n-1)a \cot (a\rho),
$$
where $\rho$ is a smooth modification of the distance function to $\partial_-\mathbb{L}$  in $(M^n, g)$ satisfying $|\nabla \rho|\leq 1$. Here we have already assume that the width of $\mathbb{L}$ is bigger than $a^{-1}\pi$. With the function $\mu: \mathbb{L}\mapsto \mathbb{R}$ defined, the remaining steps align with the proof of Theorem \ref{noncompactrigidity2}. Consequently, we complete the proof of Theorem \ref{noncompactrigidity3}.
\end{proof}

\subsection{Proof of Theorem \ref{noncompactrigidity4}}
In this subsection, we focus on rigidity of compact perturbations of $ \mathbb{S}^3\times\mathbb{R}^{n-3}$ with $R_g\geq 6$, and prove Theorem \ref{noncompactrigidity4}. With the isometry at the infinity, this case is finally reduced to Theorem \ref{comprigidity1}.
\begin{theorem}
Let $(M^n, g)$, where $4\leq n\leq 7$,  be a  noncompact orientable connected and complete Riemannian manifold with scalar curvature $R_g\geq 6 $. Let $f:M^n\to \mathbb{S}^3\times\mathbb{R}^{n-3}$ be a proper and $1$-Lipschitz  map with non-zero degree. Then $(M^n, g)$ is  isometric to $\mathbb{S}^3\times\mathbb{R}^{n-3}$ provided $f$ is isometric outside a compact domain of $M^n$.
\end{theorem}
\begin{proof}
    Since $f$ is isometric outside a compact domain $K$, we can choose a larger compact set $K_{1}\subset M$ such that $\partial K_{1}=\mathbb{S}^{3}\times (\partial [-a,a]^{n-3})$ for sufficiently large $a$.  By gluing the opposite sides of $[-a,a]^{n-3}$, we obtain a new compact manifold $\Bar{M}$ and a map $\bar{f}:\bar{M}\to \mathbb{S}^{3}\times \mathbb{T}^{n-3}$ with non-zero degree and $1$-Lipschitz property. According to Theorem \ref{comprigidity1}, $\Bar{M}$ is local isometric to $\mathbb{S}^{3}\times \mathbb{T}^{n}$. Then, in the compact set $K_{1}$, the scalar curvature is $6$ and the Ricci curvature is nonnegative. Consequently, $M$ has non-negative Ricci curvature. Finally, invoking the splitting theorem, we conclude that $(M^n, g)$ is isometric to $\mathbb{S}^3 \times \mathbb{R}^{n-3}$ due to the existence of $n-3$ geodesic lines with distinct directions.
\end{proof}


\bibliographystyle{amsplain}
\bibliography{mybib.bib}

\end{document}